\documentclass[11pt,a4paper]{article}%
\usepackage{amsmath, color, amsfonts, amsthm, latexsym}
\usepackage{amscd}
\usepackage{amsmath}
\usepackage{amsfonts}
\usepackage{amssymb}
\usepackage{graphicx}%
\providecommand{\U}[1]{\protect\rule{.1in}{.1in}}
\newtheorem{theorem}{Theorem}[section]
\newtheorem{proposition}[theorem]{Proposition}
\newtheorem{corollary}[theorem]{Corollary}

\begin{document}

\title{\textsc{Absolutely summing linear operators into spaces with no finite cotype}}
\author{Geraldo Botelho\thanks{Supported by CNPq Project 202162/2006-0.}\,\, and
Daniel Pellegrino\thanks{Supported by CNPq Projects 471054/2006-2 and
308084/2006-3.\hfill\newline2000 Mathematics Subject Classification. Primary
47B10; Secondary 46G25.}}
\date{}
\maketitle

\begin{abstract}
Given an infinite-dimensional Banach space $X$ and a Banach space $Y$ with no
finite cotype, we determine whether or not every continuous linear operator
from $X$ to $Y$ is absolutely $(q;p)$-summing for almost all choices of $p$
and $q$, including the case $p=q$. If $X$ assumes its cotype, the problem is
solved for all choices of $p$ and $q$. Applications to the theory of dominated
multilinear mappings are also provided.

\end{abstract}

\vspace*{-1.0em}

\section*{Introduction}

\indent\indent Given Banach spaces $X$ and $Y$, the question of whether or not
every continuous linear operator from $X$ to $Y$ is absolutely $(q;p)$-summing
has been the subject of several classical works, such as Bennet \cite{Ben},
Carl \cite{carl}, Dubinsky, Pe\l czy\'{n}ski and Rosenthal \cite{DPR}, Garling
\cite{gar}, Kwapie\'{n} \cite{kwa}, Lindenstrauss and Pe\l czy\'{n}ski
\cite{LP} and many others. In this note we address this question for range
spaces $Y$ having no finite cotype (such spaces are abundant in Banach space
theory). For arbitrary domain spaces $X$ the results we prove settle the
question for almost every choice of $p$ and $q$ (Theorem \ref{teorema}),
including the case $p=q$ (Corollary \ref{outro corolario}). For domain spaces
$X$ having cotype $\inf\{q:X\mathrm{~has~cotype~}q\}$ (by far most Banach
spaces enjoy this property) our results settle the question for all choices of
$p$ and $q$ (Corollary \ref{corolario}). Applications of these results to the
theory of dominated multilinear mappings are given in a final section.

\section{Background and notation}

\indent\indent Throughout this note, $n$ will be a positive integer,
$X,X_{1},...,X_{n}$ and $Y$ will represent Banach spaces over $\mathbb{K}%
=\mathbb{R}$ or $\mathbb{C}.$ The symbol $X^{\prime}$ represents the
topological dual of $X$ and $B_{X}$ the closed unit ball of $X$. The Banach
space of all continuous linear operators from $X$ to $Y$, endowed with the
usual $\sup$ norm, will be denoted by $\mathcal{L}(X;Y)$.

Given $1\leq p<+\infty$ and a Banach space $X$, the linear space of all
sequences $(x_{j})_{j=1}^{\infty}$ in $X$ such that $\Vert(x_{j}%
)_{j=1}^{\infty}\Vert_{p}:=(\sum_{j=1}^{\infty}\Vert x_{j}\Vert^{p})^{\frac
{1}{p}}<\infty$ will be denoted by $\ell_{p}(X).$ By $\ell_{p}^{w}(X)$ we
represent the linear space composed by the sequences $(x_{j})_{j=1}^{\infty}$
in $X$ such that $(\varphi(x_{j}))_{j=1}^{\infty}\in\ell_{p}$ for every
$\varphi\in X^{\prime}$. A norm $\Vert\cdot\Vert_{w,p}$ on $\ell_{p}^{w}(X)$
is defined by $\Vert(x_{j})_{j=1}^{\infty}\Vert_{w,p}:=\sup_{\varphi\in
B_{X^{\prime}}}(\sum_{j=1}^{\infty}|\varphi(x_{j})|^{p})^{\frac{1}{p}}$. A
linear operator $u\colon X\longrightarrow Y$ is said to be absolutely
$(q,p)$-summing (or simply $(q,p)$-summing), $1\leq p\leq q<+\infty$, if
$(u(x_{j}))_{j=1}^{\infty}\in\ell_{q}(Y)$ whenever $(x_{j})_{j=1}^{\infty}%
\in\ell_{p}^{w}(X).$ By $\Pi_{q;p}(X;Y)$ we denote the subspace of
$\mathcal{L}(X;Y)$ of all absolutely $(q,p)$-summing operators, which becomes
a Banach space with the norm $\pi_{q;p}(u) := \sup\{ \Vert(u(x_{j}%
))_{j=1}^{\infty}\Vert_{q} : (x_{j})_{j=1}^{\infty} \in B_{\ell_{p}^{w}(X)}
\}$. If $p=q$ we simply say that $u$ is absolutely $p$-summing (or
$p$-summing) and simply write $\Pi_{p}(X;Y)$ for the corresponding
space.\newline\indent Given a Banach space $X$, we put $r_{X}:=\inf
\{q:X\mathrm{~has~cotype~}q\}$. Clearly $2\leq r_{X}\leq+\infty$%
.\newline\indent For $1\leq p<+\infty$, $p^{\ast}$ denotes its conjugate
index, i.e., $\frac{1}{p}+\frac{1}{p^{\ast}}=1$ ($p^{\ast}=1$ if $p=+\infty
$).\newline\indent For the theory of absolutely summing operators and for any
unexplained concepts we refer to Diestel, Jarchow and Tonge \cite{Diestel}.

\section{Main results}

Henceforth $p$, $q$ and $r$ will be real numbers with $1 \leq p \leq q <
+\infty$ and $1 \leq r \leq+\infty$.

\begin{theorem}
\label{main} Let $Y$ be a Banach space with no finite cotype and suppose that
$\ell_{r}$ is finitely representable in $X.$ Then there exists a continuous
linear operator from $X$ to $Y$ which fails to be $(q;p)$-summing if either
$1\leq q<r$ or $p\geq r^{\ast}$.
\end{theorem}

\begin{proof} Assume first that $r < + \infty$. By $(e_j)_{j=1}^\infty$ we mean the canonical unit vectors of
$\ell_r$. If $1\leq q<r$, then $\left(\frac{e_j}{j^{\frac{1}{q}}} \right)_{j=1}^\infty \in \ell_1^w(\ell_r)
\subseteq \ell_p^w(\ell_r)$ because $q < r$ and $\left(\frac{e_j}{j^{\frac{1}{q}}} \right)_{j=1}^\infty
\notin \ell_q(\ell_\infty)$ (obvious). Moreover, for every $n \in \mathbb{N}$,
\[
\sup_{n}\left\| \left(  \frac{e_{j}}{j^{\frac{1}{q}}}\right) _{j=1}^{n}\right\|
_{\ell_p^w(\ell_r)}<+\infty\text{ and }\sup_{n}\left\Vert \left( \frac{e_{j}}{j^{\frac{1}{q}}}\right)
_{j=1}^{n}\right\Vert _{\ell_{q}(\ell_\infty)}=+\infty.
\]
So, for every positive integer $n$, if $u_n \colon \ell_r^n \longrightarrow \ell_\infty^n$ denotes the formal
inclusion, then
$$\sup_{n}\pi_{q;p}( u_{n})=+ \infty {\rm ~and~} \|u_n\| = 1. $$
The same is true if $p \geq r^*$ as $\left(
e_{j}\right)  _{j=1}^{\infty}\in \ell_{r^{\ast}}^{w}(\ell_{r})\subset \ell_{p}%
^{w}(\ell_{r})$ and $\left(  e_{j}\right)  _{j=1}^{\infty}\notin \ell_{q}(\ell_{\infty })$. \\
\indent We know that $\ell_{\infty}$ is finitely representable in $Y$ from the celebrated Maurey-Pisier
Theorem \cite[Theorem 11.1.14 (ii)]{Albiac} and that $\ell_{r}$ is finitely representable in
$X$ by assumption. So, for each $n \in \mathbb{N}$, there exist a subspace $Y_{n}$ of $Y$, a subspace $X_{n}$
of $X$ and linear isomorphisms $T$ and $R$
$$\ell_{\infty}^{n}\overset{T}{\longrightarrow}Y_{n}\overset{T^{-1}%
}{\longrightarrow}\ell_{\infty}^{n} {\rm ~and~}
\ell_{r}^{n}\overset{R}{\longrightarrow}X_{n}\overset{R^{-1}}{\longrightarrow }\ell_{r}^{n}$$
so that $\|T\| = \|R\| =1$, $\left\Vert T^{-1}\right\Vert <2$ and $\left\Vert R^{-1}\right\Vert <2$.
Now consider the chain%
\[
\ell_{r}^{n}\overset{R}{\longrightarrow}X_{n}\overset{R^{-1}}{\longrightarrow
}\ell_{r}^{n}\overset{u_{n}}{\longrightarrow}\ell_{\infty}^{n}\overset
{T}{\longrightarrow}Y_{n}\overset{T^{-1}}{\longrightarrow}\ell_{\infty}^{n}.%
\]
Since $\left\Vert R\right\Vert =1,$ we conclude that%
\[
\pi_{q;p}(u_n)= \pi_{q;p}( u_{n}\circ R^{-1}\circ R)\leq \pi_{q;p}(u_{n}\circ
R^{-1})\left\Vert R\right\Vert =\pi_{q;p}(u_{n}\circ R^{-1}).
\]
Hence the operator $ u_{n}\circ R^{-1} \colon X_{n}\longrightarrow l_{\infty}^{n}$
is so that%
$$\sup_{n}\pi_{q;p}(u_{n}\circ R^{-1}) =+\infty {\rm~and~} \sup_{n}\left\Vert u_{n}\circ R^{-1}\right\Vert
<+\infty.$$
Since
$\ell_{\infty}^{n}$ is an injective Banach space, there is a norm preserving
extension $v_{n} \colon X\longrightarrow \ell_{\infty}^{n}$ of $u_{n}\circ R^{-1}.$ It is immediate that%
\begin{align}
\sup_{n}\pi_{q;p}(v_{n})  =+\infty {\rm ~and~}\sup_{n}\left\Vert v_{n}\right\Vert  <+\infty.\label{aaa}
\end{align}
Consider now the operator $T\circ v_{n} \colon X\longrightarrow Y_{n}.$ Since $\left\Vert
T^{-1}\right\Vert <2,$ we have%
\begin{equation}
\pi_{q;p}(v_{n})=\pi_{q;p}(T^{-1}\circ T\circ v_{n})<2 \pi_{q;p}(T\circ v_{n}).
\label{as}%
\end{equation}
From (\ref{aaa}), (\ref{as})\ and $\left\Vert T\right\Vert =1$
we get%
\begin{align}
\sup_{n}\pi_{q;p}(T\circ v_{n}) =\infty {\rm~and~}\sup_{n} \left\Vert T\circ v_{n}\right\Vert
<+\infty.\label{aa}
\end{align}
By composing $T\circ v_{n}$ with the formal inclusion $i\colon Y_{n}\longrightarrow Y$ we obtain the operator
$ i\circ T\circ v_{n}\colon X\longrightarrow Y$. Combining the injectivity of $\Pi_{q;p}$ \cite[Proposition
10.2]{Diestel} with (\ref{aa}) we have%
\begin{align*}
\sup_{n}\left\Vert i\circ T\circ v_{n}\right\Vert _{as(q;p)}  =\infty {\rm ~and~} \sup_{n}\left\Vert i\circ
T\circ v_{n}\right\Vert  &  <\infty.
\end{align*}
Calling on the Open Mapping Theorem we conclude that $\Pi_{q;p}(X,Y)\neq \mathcal{L}(X,Y).$\\
\indent Suppose now that $\ell_\infty$ is finitely representable in $X$. Since every Banach space is finitely
representable in $c_0$, $\ell_r$ is finitely representable in $c_0$, hence in $\ell_\infty$, for every $1
\leq r < +\infty$. It follows that $\ell_r$ is finitely representable in $X$ for every $1 \leq r < +\infty$,
so the result holds for every $1 \leq r < +\infty$ by the first part of the proof, hence for $r = +\infty$.
\end{proof}

\begin{corollary}
\label{outro corolario} Regardless of the infinite-dimensional Banach space
$X$, the Banach space $Y$ with no finite cotype and $p \geq1$, there exists a
continuous linear operator from $X$ to $Y$ which fails to be $p$-summing.
\end{corollary}

\begin{proof} By Maurey-Pisier Theorem \cite[Theorem 11.3.14]{Albiac} we know that $\ell_{r_X}$ is finitely representable in $X$, so Theorem
\ref{main} provides a continuous linear operator $u \colon X \longrightarrow Y$ which fails to be $p$-summing
for every $p \geq r_X^*$. Since $\Pi_r \subseteq \Pi_s$ if $r \leq s$ \cite[Theorem 2.8]{Diestel}, it follows
that $u$ fails to be $p$-summing for every $p \geq 1$.
\end{proof}

Next result settles the question $\Pi_{q;p}(X;Y)\overset{??}{=} \mathcal{L}%
(X;Y)$ for $Y$ with no finite cotype for almost all choices of $p$ and $q$:

\begin{theorem}
\label{teorema} Let $Y$ be a Banach space with no finite cotype and $X$ be an
infinite-dimensional Banach space. Then: \newline\noindent\textrm{(a)}
$\Pi_{q;p}(X;Y)\neq\mathcal{L}(X;Y)$ if either $1\leq q<r_{X}$ or $p\geq
r_{X}^{\ast}$ or $1<p<r_{X}^{\ast}$ and $q<\frac{1}{\frac{1}{p}-\frac{1}%
{r_{X}^{\ast}}}$.\newline\textrm{(b)} $\Pi_{q;p}(X;Y)=\mathcal{L}(X;Y)$ if
either $p=1$ and $q>r_{X}$ or $1<p<r_{X}^{\ast}$ and $q>\frac{1}{\frac{1}%
{p}-\frac{1}{r_{X}^{\ast}}}$.
\end{theorem}

\begin{proof} (a) Since $\ell_{r_X}$ is finitely representable in $X$ (Maurey-Pisier Theorem),
the case $1\leq q<r_{X}$ and the case $p\geq r_{X}^{\ast}$ follow from Theorem \ref{main}. Suppose $1<p<r_{X}^{\ast}$
and $q<\frac{1}{\frac{1}{p}-\frac{1}{r_{X}^{\ast}}}$. From the previous cases we know that
$\Pi_{s;r_X^*}(X;Y)\neq \mathcal{L}(X;Y)$ for every $s \geq 1$. So the proof will be complete if we show that
$\Pi_{q;p}(X;Y) \subseteq \Pi_{s;r_X^*}(X;Y)$ for sufficiently large $s$. By \cite[Theorem 10.4]{Diestel} it
suffices to show that there exist a sufficiently large $s$ so that $q \leq s$, $r_X^* \leq s$ and
$\frac{1}{p}-\frac{1}{q}\leq\frac{1}{r_{X}^{\ast}}-\frac{1}{s}$. From
$$\frac{1}{p}-\frac{1}{q}<\frac{1}{p}-\frac{1}{\frac{1}{\frac{1}{p}-\frac
{1}{r_{X}^{\ast}}}}=\frac{1}{p}-\left(  \frac{1}{p}-\frac{1}{r_{X}^{\ast}%
}\right)  =\frac{1}{r_{X}^{\ast}}$$ we can choose $s \geq \max\{q,r_X^*\}$ such that
$\frac{1}{p}-\frac{1}{q}\leq\frac{1}{r_{X}^{\ast}}-\frac{1}{s}$, completing the proof of (a).\\
(b) If $q > r_X$, then $X$ has cotype $q$, hence the identity operator on $X$ is $(q;1)$-summing,
so $\Pi_{q;1}%
(X;Y)=\mathcal{L}(X;Y)$. Suppose $1<p<r_{X}^{\ast}$ and $q>\frac{1}{\frac{1}{p}-\frac{1}{r_{X}^{\ast}}}$.
Calling on \cite[Theorem 10.4]{Diestel} once again we have that $\Pi_{r_{X}+\varepsilon;1}(X;Y)\subset
\Pi_{q;p}(X;Y)$ for a sufficiently small $\varepsilon >0$. From the previous case we know that
$\Pi_{r_{X}+\varepsilon;1}(X;Y) = {\cal L}(X;Y)$, so $\Pi_{q;p}(X;Y) = {\cal L}(X;Y)$ as well.
\end{proof}

The only cases left open are (i) $p=1$ and $q=r_{X}$, (ii) $1<p<r_{X}^{\ast}$
and $q=\frac{1}{\frac{1}{p}-\frac{1}{r_{X}^{\ast}}}$. For spaces $X$ having
cotype $r_{X}$ the problem is completely settled:

\begin{corollary}
\label{corolario} Suppose that $Y$ has no finite cotype and that $X$ is
infinite-dimensional and has cotype $r_{X}$. Then $\Pi_{q;p}(X;Y)=\mathcal{L}%
(X;Y)$ if and only if either $p=1$ and $q\geq r_{X}$ or $1<p<r_{X}^{\ast}$ and
$q\geq\frac{1}{\frac{1}{p}-\frac{1}{r_{X}^{\ast}}}$.
\end{corollary}

\begin{proof} As mentioned above, by Theorem \ref{teorema} it suffices to consider the cases
(i) $p=1$ and $q=r_{X}$, (ii) $1<p<r_{X}^{\ast}$ and $q=\frac{1}{\frac{1}{p}-\frac{1}{r_{X}^{\ast}}}$. Since
$X$ has cotype $r_X$, the identity operator on $X$ is $(r_X;1)$-summing, so (i) is done. By \cite[Theorem
10.4]{Diestel} we have that $\Pi_{r_{X};1}(X;Y)\subset \Pi_{\frac{1}
{\frac{1}%
{p}-\frac{1}{r_{X}^{\ast}}};p}(X;Y)$ whenever $1<p<r_{X}^{\ast}$, so (ii) follows from (i).
\end{proof}

Note that Corollary \ref{corolario} improves the linear case of
\cite[Corollary 6]{StudiaD}.

The next consequence of Theorem \ref{teorema}, which is closely related to a
classical result of Maurey-Pisier \cite[Remarque 1.4]{MP} and to \cite[Example
2.1]{Dan}, shows that fixed an infinite-dimensional Banach space $X$, the
number $\inf\{q : \Pi_{q;1}(X;Y)=\mathcal{L}(X;Y)\}$ does not depend on the
Banach space with no finite cotype $Y$.

\begin{corollary}
Let $X$ be an infinite-dimensional Banach space. Then $r_{X} = \inf\{q :
\Pi_{q;1}(X;Y)=\mathcal{L}(X;Y)\}$ regardless of the Banach space $Y$ with no
finite cotype.
\end{corollary}

\section{Applications to the multilinear theory}

\indent\indent One of the most interesting and most studied multilinear
generalizations of the ideal of absolutely $p$-summing linear operators is the
class of $p$-dominated multilinear mappings. A continuous $n$-linear mapping
$A\colon X_{1}\times\cdots\times X_{n}\longrightarrow Y$ is $(p_{1}%
,\ldots,p_{n})$-dominated, $1\leq p_{1},\ldots,p_{n}<+\infty$, if
$(A(x_{j}^{1},\ldots,x_{j}^{n}))_{j=1}^{\infty}\in\ell_{q}(Y)$ whenever
$(x_{j}^{k})_{j=1}^{\infty}\in\ell_{p_{k}}^{w}(X_{k})$, $k=1,\ldots,n$, where
$\frac{1}{q}=\frac{1}{p_{1}}+\cdots+\frac{1}{p_{n}}$. If $p_{1}=\cdots
=p_{n}=p$ we simply say that $A$ is $p$-dominated. For details we refer to
\cite{BP-proc,Anais}.\newline\indent Continuous bilinear forms on either an
$\mathcal{L}_{\infty}$-space, or the disc algebra $\mathcal{A}$ or the Hardy
space $H^{\infty}$ are 2-dominated \cite[Proposition 2.1]{BP-proc}. On the
other hand, partially solving a problem posed in \cite{BP-proc}, in
\cite[Lemma 5.4]{Jarc} it was recently shown that for every $n\geq3$, every
infinite-dimensional Banach space $X$ and any $p\geq1$, there is a continuous
$n$-linear form on $X^{n}$ which fails to be $p$-dominated. As to
vector-valued bilinear mappings, all that is known, as far as we know, is that
for every $\mathcal{L}_{\infty}$-spaces $X_{1},X_{2}$, every
infinite-dimensional space $Y$ and any $p\geq1$, there is a continuous
bilinear mapping $A:X_{1}\times X_{2}\rightarrow Y$ which fails to be
$p$-dominated \cite[Theorem 3.5]{bot97}. Besides of giving an alternative
proof of \cite[Lemma 5.4]{Jarc}, we fill in this gap concerning vector-valued
bilinear mappings by generalizing \cite[Theorem 3.5]{bot97} to arbitrary
infinite-dimensional spaces $X_{1},X_{2},Y$.

\begin{proposition}
Let $X_{1}, X_{2}$ and $Y$ be infinite-dimensional Banach spaces and let
$p_{1}, p_{2} \geq1$. Then there exists a continuous bilinear mapping $A
\colon X_{1} \times X_{2} \longrightarrow Y$ which fails to be $(p_{1},p_{2})$-dominated.
\end{proposition}

\begin{proof} Suppose, by contradiction, that every continuous bilinear mapping from $X_1 \times X_2$ to $Y$
is $(p_1,p_2)$-dominated. A straightforward adaptation of the proof of \cite[Lemma 3.4]{bot97} gives that
every continuous linear operator from $X_1$ to ${\cal L}(X_2;Y)$ is $p_1$-summing. From \cite[Proposition
19.17]{Diestel} we know that ${\cal L}(X_2;Y)$ has no finite cotype, so Corollary \ref{outro corolario}
assures that there is a continuous linear operator from $X_1$ to ${\cal L}(X_2;Y)$ which fails to be
$p_1$-summing. This contradiction completes the proof.
\end{proof}

The same reasoning extends \cite[Lemma 5.4]{Jarc} to $(p_{1}, \ldots, p_{n}%
)$-dominated $n$-linear mappings (for eventually different $p_{1}, \ldots,
p_{n}$) on $X_{1} \times\cdots\times X_{n}$ (for eventually different spaces
$X_{1}, \ldots, X_{n}$):

\begin{proposition}
Let $n\geq3$, $X_{1},\ldots,X_{n}$ be Banach spaces at least three of them
infinite-dimensional and let $p_{1},\ldots,p_{n}\geq1$. Then there exists a
continuous $n$-linear form $A\colon X_{1}\times\cdots\times X_{n}%
\longrightarrow\mathbb{K}$ which fails to be $(p_{1},\ldots,p_{n})$-dominated.
\end{proposition}

\textbf{Acknowledgement.} The authors thank Joe Diestel for helpful
conversations on the subject of this paper.

\vspace*{1em} \noindent[Geraldo Botelho] Faculdade de Matem\'atica,
Universidade Federal de Uberl\^andia, 38.400-902 - Uberl\^andia, Brazil,
e-mail: botelho@ufu.br.

\medskip

\noindent[Daniel Pellegrino] Departamento de Matem\'atica, Universidade
Federal da Pa-ra\'iba, 58.051-900 - Jo\~ao Pessoa, Brazil, e-mail: dmpellegrino@gmail.com.

\end{document}